\documentclass[11pt]{amsart}

\usepackage{epigamath}

\usepackage[english]{babel}

\numberwithin{equation}{section}

\newtheorem{theorem}{Theorem}[section]
\newtheorem{corollary}[theorem]{Corollary}
\newtheorem{lemma}[theorem]{Lemma}
\newtheorem{proposition}[theorem]{Proposition}

\theoremstyle{definition}

\newtheorem{remark}[theorem]{Remark}
\newtheorem{notation}[theorem]{Notation}
\newtheorem*{acknowledgements}{Acknowledgements}

\newcommand{\thmref}[1]{Theorem~\ref{#1}}

\newcommand{\lemref}[1]{Lemma~\ref{#1}}
\newcommand{\propref}[1]{Proposition~\ref{#1}}
\newcommand{\corref}[1]{Corollary~\ref{#1}}

\DeclareFontFamily{OT1}{rsfs}{}
\DeclareFontShape{OT1}{rsfs}{n}{it}{<-> rsfs10}{}
\DeclareMathAlphabet{\curly}{OT1}{rsfs}{n}{it}

\newcommand\LL{\mathbf L}

\renewcommand\O{\mathcal O}

\newcommand\cE{\mathcal E}
\newcommand\F{\mathcal F}

\newcommand\Coeff{\mathrm{Coeff}}

\newcommand\td{\mathrm{td}}

\newcommand\C{\mathbb C}

\newcommand\Q{\mathbb Q}

\newcommand\R{\mathbb R}

\newcommand\Z{\mathbb Z}

\newcommand\INTO{\ar@{^{(}->}[r]}

\newcommand\rk{\operatorname{rk}}

\newcommand\ch{\operatorname{ch}}

\newcommand\vd{\operatorname{vd}}

\newcommand\Pic{\operatorname{Pic}}

\newcommand\Sym{\operatorname{Sym}}
\newcommand\beq[1]{\begin{equation}\label{#1}}
\newcommand\eeq{\end{equation}}
\newcommand\beqa{\begin{eqnarray*}}
\newcommand\eeqa{\end{eqnarray*}}
\newcommand\<{\langle}
\renewcommand\>{\rangle}
\DeclareMathOperator{\Ell}{Ell}
\DeclareMathOperator{\ELL}{ELL}
\DeclareMathOperator{\num}{num}
\DeclareMathOperator{\orb}{orb}

\EpigaVolumeYear{4}{2020} \EpigaArticleNr{15}
\ReceivedOn{March 13, 2019}
\InFinalFormOn{July 17, 2020}
\AcceptedOn{August 19, 2020}

\title{Refined Verlinde formulas for Hilbert schemes of points and moduli spaces of sheaves on K3 surfaces}
\titlemark{Refined Verlinde formulas}

\author{Lothar~G\"ottsche}
\address{International Centre for Theoretical Physics, Strada Costiera 11, 34100 Trieste, Italy}
\email{gottsche@ictp.it}

\authormark{L.~G\"ottsche}

\AbstractInEnglish{We compute generating functions for elliptic genera with values in  line bundles on  Hilbert schemes of points on surfaces. 
As an application we also compute generating functions for elliptic genera with values in  determinant line bundles on moduli spaces of sheaves on K3 surfaces.}



\TitleInFrench{Formules de Verlinde raffin\'ees pour les sch\'emas de Hilbert de points et espaces de modules sur les surfaces K3}

\AbstractInFrench{Nous calculons les fonctions g\'en\'eratrices pour les genres elliptiques \`a valeur dans les fibr\'es en droites sur les sch\'emas de Hilbert des points sur les surfaces. En guise d'application, nous calculons \'egalement les fonctions g\'en\'eratrices pour les genres elliptiques \`a valeurs dans les fibr\'es en droites sur les espaces de modules de faisceaux sur les surfaces K3.}

\begin{document}



\maketitle

\begin{prelims}

\DisplayAbstractInEnglish

\bigskip

\DisplayKeyWords

\medskip

\DisplayMSCclass

\bigskip

\languagesection{Fran\c{c}ais}

\bigskip

\DisplayTitleInFrench

\medskip

\DisplayAbstractInFrench

\end{prelims}


\newpage

\setcounter{tocdepth}{2}

\tableofcontents


\section{Introduction} 
The celebrated Verlinde formula (see \cite{Ver,NR,BL,Fal}) is a formula for the generating function for   dimensions of spaces of sections of line bundles on moduli spaces of vector bundles on algebraic curves.

Now let $S$ be a smooth projective surface over $\C$ and $S^{[n]}$ the Hilbert scheme of $n$ points on $S$. 
For every vector bundle $V$ on $S$ there is a corresponding tautological bundle $V^{[n]}$ of rank $\rk(V^{[n]})=n\rk(V)$, whose fibre
over $Z\in S^{[n]}$ is $H^0(Z,V|_Z)$. The map $V\mapsto V^{[n]}$ extends to a homomorphism from the Grothendieck group $K^0(S)$ of vector bundles on $S$ to $K^0(S^{[n]})$.
For $L\in S^{[n]}$ denote $\mu(L):=\det((L-\O_S)^{[n]})$ and $E:=\det(\O^{[n]}_S)$.
Then it is well known that $\Pic(S^{[n]})=\mu(\Pic(S))\oplus \Z E$.
The analogue of the Verlinde formula for Hilbert schemes of points is a formula for the generating function for  holomorphic Euler characteristics $\chi(S^{[n]},\mu(L)\otimes E^r)$.
In \cite{EGL} such a formula was proven in the cases $r=-1,0,1$ or 
$K_S^2=K_SL=0$.
On the other hand the celebrated Dijkgraaf-Moore-Verlinde-Verlinde formula \cite{DMVV}, shown in \cite{BL1,BL2,BL3}, relates the generating function of the elliptic genera $\Ell(S^{[n]})$ of Hilbert schemes of points to Siegel modular forms.

In this short note we interpolate between these two results,  by proving a formula for $\Ell(S^{[n]},\mu(L)\otimes E^r)$, the elliptic genus with values in the line bundle $\mu(L)\otimes E^r$.
To state these results we introduce the following power series.
\begin{align*}
\phi_{-2,1}(q,y)&:=(y^{1/2}-y^{-1/2})^2\prod_{n>0} \frac{(1-q^ny )^2(1-q^n/y)^2}{(1-q^n)^4},\\
\wp(q,y)&:=\frac{1}{12}+\frac{1}{(y^{1/2}-y^{-1/2})^2}+ \sum_{n>0} \sum_{d|n} d (y^d-2+y^{-d}) q^n,\\
\phi_{0,1}(q,y)&:=12\wp(x,y)\phi_{-2,1}(x,y),\\
h(q,y)&:=-\frac{1}{12}+\sum_{n>0}\sum_{d|n} \frac{n}{d} (y^d+y^{-d}) x^n.
\end{align*}
$\wp(q,y)$ is the Weierstrass $\wp$-function.
It is standard that
$$\Ell(S)=-(\wp+h)\phi_{-2,1}K_S^2+\phi_{0,1}\chi(\O_S).$$
We also introduce the following Borcherds type lifts. For $f:=\sum_{m,n}c_{m,n}q^my^n\in \Q[[y,q]]$, we put
\begin{align*}
\LL(f,p)&:=\prod_{l>0,m\geq 0,n \in \Z} (1-p^{l} q^m y^n)^{c_{lm,n}},\\ 
\LL^{(a,b)}(f,p)&:=\prod_{l>0,m\geq 0,n \in \Z} (1-p^{l} q^m y^n)^{l^a n^b c_{lm,n}},\quad a,b\in \Z.
\end{align*}
Then the DMVV formula says 
$$\sum_{n\ge 0} \Ell(S^{[n]}) p^n=\frac{1}{\LL(\Ell(S),p)}=\frac{\LL((\wp+h)\phi_{-2,1},p)^{K_S^2}}{\LL(\phi_{0,1},p)^{\chi(\O_S)}}.$$


The first theorem deals with the case $r=0$.
\begin{theorem} \label{ellr0}
Let $S$ be an algebraic surface, $L\in \Pic(S)$.
Then 
\begin{align*}\sum_{n\ge 0} \Ell(S^{[n]},\mu(L)) p^n=&\frac{\big(\LL^{(2,0)}(-\phi_{-2,1},p)\big)^{\frac{L^2}{2}}\big(\LL^{(1,1)}(-\phi_{-2,1},p)\big)^{\frac{LK_S}{2}}}{
\LL(\Ell(S),p)}.
\end{align*}
\end{theorem}
Specializing  to $L=\O_S$, we recover the DMVV formula. Specializing to $q=0$ yields an infinite product formula for the $\chi_y$-genera with values in a $\mu(L)$, which  in turn recovers for $L=\O_S$
the formula for the $\chi_{-y}$-genera of Hilbert schemes from \cite{GS}.
We write 
$$\widetilde\Delta(p,y):=\prod_{n>0} (1-p^n)^{20}(1-p^ny)^2(1-p^n/y)^2, \quad \overline \eta(p)=\prod_{n>0}(1-p^n).$$

\begin{corollary}\label{ellch}
\begin{align*}
\sum_{n\ge 0} \overline \chi_{-y}(S^{[n]},\mu(L))p^n&=\prod_{n>0} \left(\frac{(1-p^n)^2}{(1-p^ny)(1-\frac{p^n}{y})}\right)^{\frac{n^2}{2}L^2}
\prod_{n>0}\left(\frac{1-\frac{p^n}{y}}{1-p^ny}\right)^{\frac{n}{2}LK_S}\frac{\overline \eta(p)^{K_S^2}}{\widetilde\Delta(p,y)^{\frac{\chi(\O_S)}{2}}}.
\end{align*}
\end{corollary}





For general line bundles on Hilbert schemes of points  we can partially determine the generating function.

\begin{theorem}\label{ellrr}
For every $r\in \Z$ there are universal power series
$A_r, B_r\in \Q[y^{\pm 1}][[z,q]]$ such that for every smooth projective surface $S$ and every $L\in \Pic(S)$ we have
\begin{align*}\sum_{n\ge 0} \Ell(S^{[n]},\mu(L)\otimes E^r) p^n&=
 \frac{\LL^{(2,0)}(-\phi_{-2,1},z)^{\frac{1}{2}(L^2+r^2\chi(\O_S))}A_r^{LK_S} B_r^{K_S^2}}{\left(\LL(2\phi_{0,1},z)\left(1-r^2z\frac{\partial}{\partial z}\log \LL^{(2,0)}(\phi_{-2,1},z)\right)\right)^{\frac{\chi(\O_S)}{2}}},
\end{align*}
with the change of variables 
$p=z \LL^{(2,0)}(-\phi_{-2,1},z)^{r^2}$.
\end{theorem}
Specialising to $q=0$ again yields a formula for the $\chi_{-y}$-genus with values in $\mu(L)\otimes E^r$.
\begin{corollary}
For every smooth projective surface $S$ and every $L\in \Pic(S)$ we have
\begin{align*}\sum_{n\ge 0} \overline\chi_{-y}(S^{[n]},\mu(L)\otimes E^r) p^n&=\prod_{n>0} \left(\frac{(1-z^n)^2}{(1-z^ny)(1-z^n/y)}\right)^{\frac{n^2}{2}(L^2+r^2\chi(\O_S))}\\&\times
\frac{A_r(y,z,0)^{LK_S} B_r(y,z,0)^{K_S^2}}{\left(\widetilde\Delta(z,y)\left(1+r^2\sum_{n\ge 1}(y^{d}-2+y^{-d})\frac{n^3}{d^3} z^n\right)\right)^{\frac{\chi(\O_S)}{2}}}.
\end{align*}
with the change of variables 
$p=z \prod_{n>0} \left(\frac{(1-z^n)^2}{(1-z^ny)(1-z^n/y)}\right)^{n^2r^2}$.
\end{corollary}

Finally we show an analogue of \thmref{ellrr} for moduli spaces of sheaves on K3 surfaces $S$.
Fix $s\in \Z_{> 0}$. Let $H$ be an ample line bundle on $S$ and 
$M_S^H(s,c_1,c_2)$ the moduli space of $H$-semistable sheaves on $S$ of rank $s$ with Chern classes $c_1$, $c_2$.
We assume that $M_S^H(s,c_1,c_2)$ consists only of stable sheaves. 
For any choice of $s,c_1,c_2$, we denote $$\vd=\vd(s,c_1,c_2):=2sc_2-(s-1)c_1^2-2(s^2-1).$$
Let $r\in \Z$. Let $L\in \Pic(S)\otimes \Q$ with $c_1(L)-\frac{r}{s}c_1\in H^2(S,\Z)$.
 If $s$ divides $(c_1(L)-\frac{r}{s}c_1)c_1+r(\frac{c_1^2}{2}-c_2)$, we define a determinant line bundle  $\mu(L)\otimes E^{r}\in Pic(M^H_S(s,c_1,c_2))$.
 This is
the generalization of the line bundle with the same name on $S^{[n]}=M^H_S(1,0,n)$.
We obtain the following result.

\begin{theorem}\label{ellK3} Let $S$ be a K3 surface, $s\in \Z_{>0}$. Under the assumptions above we have
\begin{align*}\Ell(M^H_S(s,c_1,c_2),\mu(L)\otimes E^r)&=\Coeff_{p^{\vd/2}}
\left[\frac{\LL^{(2,0)}(-\phi_{-2,1},z)^{\frac{L^2}{2}+\frac{r^2}{s^2}}}{\LL(2\phi_{0,1},z)\left(1-\frac{r^2}{s^2}z\frac{\partial}{\partial z}\log \LL^{(2,0)}(\phi_{-2,1},z)\right)}\right],
\end{align*}
with the change of variables 
$p=z \LL^{(2,0)}(-\phi_{-2,1},z)^{\frac{r^2}{s^2}}$,  and
\begin{align*} \overline\chi_{-y}(M^H_S(s,c_1,n),\mu(L)\otimes E^r)&=\Coeff_{p^{\vd/2}}\left[
\frac{\prod_{n>0} \left(\frac{(1-z^{n})^2}{(1-z^{n}y)(1-z^{n}/y)}\right)^{n^2(\frac{L^2}{2}+\frac{r^2}{s^2})}}{\left(\widetilde\Delta(z,y)\left(1+\frac{r^2}{s^2}\sum_{n\ge 1}\sum_{d|n} (y^{d/2}-y^{-d/2})^2\frac{n^3}{d^3} z^{n}\right)\right)}\right],
\end{align*}
with the change of variables 
$p=z \prod_{n>0} \left(\frac{(1-z^{n})^2}{(1-z^{n}y)(1-z^{n}/y)}\right)^{\frac{n^2r^2}{s^2}}$.
\end{theorem}

In the special case of K3 surfaces \thmref{ellK3}  in particular confirms the conjectures of \cite{GKW} about refinements of Verlinde formulas for moduli spaces of rank 2 sheaves on surfaces in the case of K3 surfaces.
The specialization $L=\O_X$ reproduces in the case of K3 surfaces the formulas of  \cite {GK} on the elliptic genus of moduli spaces of sheaves on surfaces.

\begin{acknowledgements}
I thank Don Zagier for helping me with the proof of \lemref{Lagrange}. This work grew out of collaboration with Martijn Kool. I thank him for many useful discussions.
\end{acknowledgements}

\section{Background material}
\subsection{Hilbert schemes of points}
Let $S$ be a smooth projective surface. We denote $S^{[n]}$ the Hilbert scheme of $n$ points on $S$. It is a smooth projective variety of dimension $2n$.
Let $S^{(n)}$ be the $n$-th symmetric power of $S$. The Hilbert-Chow morphism $\pi:S^{[n]}\to S^{(n)}, Z\mapsto supp(Z)$, sending a zero dimensional scheme to its support with multiplicities is a crepant resolution of $S^{(n)}$, i.e. it is birational and $\pi^*K_{S^{(n)}}=K_{S^{[n]}}$.
Let $Z_n(S)\subset S\times S^{[n]}$ be the universal subscheme, with projections $p:Z_n(S)\to S^{[n]}$, $q:Z_n(S)\to S$.
For a vector bundle $V$ of rank $r$ on $S$ the corresponding tautological vector bundle  is
$V^{[n]}:=p_*q^*V$, a vector bundle of rank $rn$ on $S^{[n]}$. This extends to a homomorphism $\vphantom{S}^{[n]}:K^0(S)\to K^0(S^{[n]})$ between the
Grothendieck groups of vector bundles. 
We put $E:=\det(\O_S^{[n]})$, and for a line bundle $L$ on $S$ we put $\mu(L):=\det((L-\O_S)^{[n]})$.
Let $\eta:S^n\to S^{(n)}$ be the natural projection. 
Let $L_n:=\eta_*(\otimes_{i=1}^n pr_i^*L)^{\mathfrak S_n}$ be the $\mathfrak S_n$-equivariant pushforward, where $pr_i:S^n\to S$ is the $i$-th projection. 
Then it is well-known that $\mu(L)=\eta^*(L_n)$, and from the definitions it follows that $\det(V^{[n]})=\mu(\det(V))\otimes E^{\rk(V)}$.

\subsection{Elliptic genus.}
For a compact complex manifold $M$, the $\chi_{-y}$-genus is $$\chi_{-y}(M):=\sum_p (-y)^p \chi(M,\Omega^p_M).$$ 
Usually we consider the normalized version
$\overline \chi_{-y}(M):=(-y)^{-\frac{\dim(M)}{2}}\chi_{-y}(M)$.
For $V\in K^0(M)$ the $\chi_{-y}$-genus with values in $V$ is 
$$\overline \chi_{-y}(M,V):=y^{-\frac{\dim(M)}{2}}\sum_p (-y)^p \chi(M,\Omega^p_M\otimes V).$$
For a rank $r$ vector bundle $V$ on $M$ put
\begin{align*}
\Lambda_t V := \sum_{n=0}^{r} [\Lambda^n V] \, t^n, \quad \Sym_t V := \sum_{n=0}^{\infty} [\Sym^n V] \, t^n.
\end{align*}
Write $y=e^{2\pi i z}$, $q=e^{2\pi i \tau}$.
Then for $W\in K^0(M)$,  the elliptic genus and the  elliptic genus with values in $W$ are defined by
\begin{align*}
\Ell(M) &:= \Ell(M,z,\tau):=\overline \chi_{-y}(M, \cE(T_M)), \\
\Ell(M,W)& :=\Ell(M,W,z,\tau):= \overline \chi_{-y}(M, \cE(T_M)\otimes W),
\end{align*}
with 
\begin{align*}
\cE(V) &:= \bigotimes_{n=1}^{\infty} \Lambda_{-yq^n} V^{\vee} \otimes \Lambda_{-y^{-1}q^n} V \otimes \Sym_{q^n} (V \oplus V^{\vee}).
\end{align*}
Let 
$$\theta(z,\tau):=-iq^{\frac{1}{8}}(y^{\frac{1}{2}}-y^{-\frac{1}{2}})\prod_{n>0}(1-q^n)(1-q^ny)(1-q^n/y)$$
be the classical Jacobi theta function, where we write $y=e^{2\pi i z}$, $q=e^{2\pi i \tau}$.
Let $c(M)=\prod_{j=1}^n(1+x_j)$ be a formal splitting of the total Chern class of $M$. Putting
$$\ELL(M):=\ELL(M,z,\tau)=\prod_{j} x_j\frac{\theta(\frac{x_j}{2\pi i}-z, \tau)}{\theta(\frac{x_j}{2\pi i}, \tau)}\in H^*(X,\Q)[y^{\pm \frac{1}{2}}][[q]],$$ it
follows from the definitions and the Riemann-Roch theorem that 
$$\Ell(M)=\int_M \ELL(M), \quad \Ell(M,W)=\int_M \ELL(M)\ch(W).$$

\subsection{Beauville-Bogomolov quadratic form} Let $X$ be a compact holomorphic symplectic manifold of dimension $\dim(X)=2m$.
We briefly review some properties of  the Beauville-Bogomolov quadratic form $q_X:H^2(X)\to \Q$ on $X$ from \cite[Sections~1.9--1.11]{H}.
 Note that  the odd Chern classes of $X$ vanish.
  \begin{theorem}\label{BBRR}
For any $\beta\in H^{4k}(X,\Q)$ in the sub-algebra generated by the Chern classes of $X$, there is a constant $c(\beta)\in \Q$, such that for all $\alpha\in H^2(X,\Q)$
$$\int_X \beta\alpha^{\dim(X)-2k}=c(\beta)q_X(\alpha)^{m-k}.$$
The quantity $c(\beta)$ is invariant under deformation of $X$.
\end{theorem}
\begin{corollary}\label{BBcor}
There exists a polynomial $h_X=h_X(z)$ with coefficients in $\Q[y^{\pm 1}][[q]] $ such that 
$$ \int_{X} \ELL(X)\exp(\alpha)=h_X(q_{X}(\alpha)), \quad \hbox{for all }\alpha\in H^2(X,\Z).$$
The polynomial $h$ is invariant under deformation of $X$.
\end{corollary}

\section{The case $r=0$.}
In this section we will prove \thmref{ellr0}.
We start by reviewing some of the ideas and definitions of \cite{BL3}. 
For a Kawamata log-terminal pair $(Z,D)$ of a projective variety and a divisor $D$ in $Z$ with an action of a finite group $G$,
Borisov and Libgober define in \cite[Definition~3.6]{BL3} the orbifold elliptic class $\Ell_{\orb}(Z,D,G)=\Ell_{\orb}(Z,D,G,z,\tau)$.
In fact they first define it in \cite[Definition~3.2]{BL3}  in case $X$ is nonsingular  and the pair $(X,E)$ is also $G$-normal (see \cite[Definition~3.1]{BL3} for the definitions) by an explicit formula.
In the general case, they define  in \cite[Definition~3.6]{BL3} 
$\Ell_{\orb}(Z,D,G):=\rho_*\Ell_{\orb}(X,E,G)$ for $\rho:(X,E)\to (Z,D)$ a $G$-normal equivariant resolution.
We will write $\Ell_{\orb}(Z,G)$, $\Ell_{\orb}(Z,D)$, $\Ell_{\orb}(Z)$ in case $D=0$ and/or  $G$ is the trivial group.

If $X$ is a nonsingular projective variety, with an action of a finite group 
their formula specializes to 
\begin{align}\label{ellorb}
\Ell_{\orb}(X,G):=\frac{1}{|G|}\sum_{gh=hg}\sum_{Z}[Z] \Big(\prod_{\lambda(g)=\lambda(h)=0} \prod_{j} x_{\lambda,j}\Big)
\prod_{\lambda,j}\frac{\theta(\frac{x_{\lambda,j}}{2\pi i} +\lambda(g)-\tau \lambda(h)-z)}{\theta(\frac{x_{\lambda,j}}{2\pi i } +\lambda(g)-\tau\lambda(h))}e^{2\pi i \lambda(h)z}.
\end{align}
Here $Z$ runs over the irreducible components of the common fixpoint set of $g$ and $h$, and $[Z]$ is the class of $Z$ in the Chow group of $X$.
The restriction of $T_X$ to Z splits into linearized bundles according to the $[0,1)$-valued characters $\lambda$ of $\<g,h\>$. We denote by $x_{\lambda,j}$ the elements of a formal splitting of the  total Chern class of the bundle with character $\lambda$.
If $G$ acts effectively on $(Z,D)$ and   $(Z/G,D/G)$ is the quotient pair in the sense of \cite[Definition~2.7]{BL3}, they show in \cite[Theorem.~5.3]{BL3} that 
$\psi_*\Ell_{\orb}(Z,D)=\Ell_{\orb}(Z/G,D/G)$ for the quotient morphism $\psi:Z\to Z/G$. This is in particular true for the pairs $(Z,0)$, $(Z/G,0)$, if $Z$ is nonsingular and  $G$ is   acting freely in  codimension $1$.

Now consider the action of $\mathfrak S_n$ on $S^n$, and recall the quotient morphism $\eta:S^n\to S^{(n)}$ and the Hilbert Chow morphism $\pi:S^{[n]}\to S^{(n)}$. 
As $S^{[n]}$ is nonsingular we have $\ELL(S^{[n]})=\Ell_{\orb}(S^{[n]})$. As $\pi$ is a crepant resolution, \cite[Theorem~3.5]{BL3} implies that $\pi_*\ELL(S^{[n]})=\Ell_{\orb}(S^{(n)})$.
Thus we find by the above
\begin{equation}\label{BLeq}
\pi_*\ELL(S^{[n]})=\Ell_{\orb}(S^{(n)})=\eta_* \Ell_{\orb}(S^n,\mathfrak S_n).
\end{equation}
Now let $L$ be a line bundle on $S$, then we have
\begin{align*}
\Ell(S^{[n]},\mu(L))&=\int_{S^{[n]}}\ELL(S^{[n]})\ch(\mu(L))\\&=\int_{S^{(n)}}\pi_*\big(\ELL(S^{[n]})\big)\ch(L_n)\\
&=\int_{S^n}\Ell_{\orb}(S^n,\mathfrak S_n)  \ch\Big(\sum_{i=1}^n pr_i^* L\Big).
\end{align*}
In the second line we have used $\pi^*(L_n)=\mu(L)$ and  the projection formula,  and in the third line we have used \eqref{BLeq}, $\eta^*(L_n)=\otimes_{i=1}^n pr_i^*L$ and again the projection formula.

Let  $x_1,x_2$ be the Chern roots of $T_S$. Then
\begin{equation}\label{ellsl}
\begin{split}
&\Ell(S,L)=\int_S (1+L+\frac{L^2}{2})\prod_{j=1}^2 x_i\frac{\theta(\frac{x_j}{2\pi i}-z,\tau)}{\theta(\frac{x_j}{2\pi i},\tau)} \\
&=\Ell(S)-\frac{K_SL}{2} \Coeff_{x^1}\left(x\frac{\theta(\frac{x}{2\pi i}-z,\tau)}{\theta(\frac{x}{2\pi i},\tau) }\right)^2
+\frac{L^2}{2}\Coeff_{x^2}\left(x\frac{\theta(\frac{x}{2\pi i}-z,\tau)}{\theta(\frac{x}{2\pi i},\tau) }\right)^2\\
&=
\Ell(S)+\frac{K_SL}{2}y\frac{\partial}{\partial y} \phi_{-2,1} +\frac{L^2}{2}\phi_{-2,1}.
\end{split}
\end{equation}

Now we prove  \thmref{ellr0} by adapting the proof of \cite[Theorem~6.1]{BL3}. Note that in the notations of \cite{BL3} we have $D=0$, which  leads to many simplifications.

Let $(g,h)$ be a commuting pair in $\mathfrak S_n$. We sum up the description of the action of $g,h$ and their fixpoint sets in the proof of \cite[Theorem~6.1]{BL3}. 
We have a decomposition $\{1,\ldots,n\}=J_1\cup \ldots\cup J_m$ into the orbits of the subgroup generated by $g,h$. Thus 
the action of $(g,h)$ on $S^n$ restricts to an action on each of the corresponding products  $S^{J_l}$.
Furthermore we can write $|J_l|=a_lb_l$ for positive integers $a_l,b_l$, and up to reordering of the elements of $J_l$ the action of $(h,g)$ on $S^{J_l}$ can be described as follows. Write  $(y_{i,j})_{i\in \Z/a_lZ,\ j\in \Z/b_l\Z}$ the components of elements
of $y\in S^{J_l}$. Then the action of $g,h$ on $S^{J_l}$ is given by $h(y_{i,j})=y_{i,j+1}$, $g(y_{i,j})=y_{i+1,j}$ for $0\le i<a_l-1$, $g(y_{a_l-1,j})=y_{0,j+s}$, for some $s\in \{0,\ldots,b_l-1\}$, and $s$ determines the 
action of $(g,h)$ on $J_l$ uniquely.
The fixpoint set $(S^{J_l})^{g,h}$ is $S$ embedded via the diagonal map $j_l:S\to S^{J_l}$.

Changing their notation slightly, we  denote for $(a,b,s):=(a_l,b_l,s_l)$ by
$$F_{a,b,s}:= j^*_{l}\left(\Big(\prod_{\lambda(g)=\lambda(h)=0}\prod_{j} x_{\lambda,j}\Big) \prod_{\lambda,j} \frac{\theta(\frac{x_{\lambda,j}}{2\pi i} +\lambda(g)-\tau \lambda(h)-z)}{\theta(\frac{x_{\lambda,j}}{2\pi i } +\lambda(g)-\tau\lambda(h))}e^{2\pi i \lambda(h)z}\right)$$
 the pullback of the contribution of the restriction of the pair $(f,g)$  to $\Ell_{\orb}(S^{J_l},\mathfrak S_{J_l})$ in  \eqref{ellorb}
 multiplied by $a!b!$.
Then it is shown in \cite[Lemma~6.4]{BL3} that $$F_{a,b,s}=\prod_{j=1,2} \frac{x_j\theta(\frac{ax_j}{2\pi i } -az,\frac{a\tau-s}{b})}{\theta(\frac{ax_j}{2\pi i },\frac{a\tau-s}{b})}.$$
Note that the left hand side is 
$$\frac{1}{a^2}\prod_{j=1,2} \frac{ax_j\theta(\frac{ax_j}{2\pi i } -az,\frac{a\tau-s}{b})}{\theta(\frac{ax_j}{2\pi i },\frac{a\tau-s}{b})}=\sum_{k=0}^2 a^{k-2}\ELL(S,az,\frac{a\tau-s}{b})_k,$$
where $(\ )_k$ denotes the part in degree $k$. As $j_l^*\eta^*(L_n)=L^{ab}$, we obtain
\begin{align*}\int_S F_{a,b,s} \ch(j_l^*\eta^*(L_n))
&=
\int_S\Big( \sum_{k=0}^2 a^{k-2}\ELL(S,az,\frac{a\tau-s}{b})_k (ab)^{2-k} \ch_{2-k}(L)\Big)\\&=
\Ell(S,L^b,az,\frac{a\tau-s}{b}).
\end{align*}
By definition it is clear that the contribution of $(g,h)$ to 
$\int_{S^{n}} \Ell_{\orb}(S^n,\mathfrak S_n)\eta^*L_n$ is 
$$\frac{1}{n!}\prod_{l=1}^m \int_SF_{a_l,b_l,s_l}j_l^*\eta^*(L_n).$$
Thus arguing as after \cite[Lemma~6.6]{BL3}, writing  
$$\Ell(S^{[n]})=\sum_{l,n}c_{l,n}y^lq^n, \quad \phi_{-2,1}=\sum_{l,n} d_{l,n}y^lq^n,$$ and
using \eqref{ellsl} in the third line,
 we obtain
\begin{align*}
\sum_{n\ge 0} \Ell&(S^{[n]},\mu(L)) p^n=\sum_{n\ge 0} p^n
\Big(\int_{S^n}\Ell_{\orb}(S^n,\mathfrak S_n)\eta^*L_n\Big)\\
&=
\exp\left(\sum_{a,b>0}\sum_{s=0}^{b-1}\frac{p^{ab}}{ab} \Ell(X,L^b,az,\frac{a\tau-s}{b})\right)\\
%
&=\exp\left(\sum_{a,b>0}\sum_{m,l}\sum_{s=0}^{b-1}\Big(c_{m,l}+\Big(bl\frac{LK_S}{2}+b^2\frac{L^2}{2}\Big)d_{m,l}\Big)\frac{p^{ab}}{ab}y^{al}q^{\frac{am}{b}}e^{2\pi i\frac{ms}{b}}\right)\\
&=\exp\left(\sum_{a,b>0}\sum_{m,l}\Big(c_{mb,l}+\Big(bl\frac{LK_S}{2}+b^2\frac{L^2}{2}\Big)d_{mb,l}\Big)\frac{p^{ab}}{a}y^{al}q^{am}\right)\\
&=
\prod_{b=1}^\infty\prod_{m,l} (1-p^by^lq^m)^{-c_{mb,l}-lb d_{mb,l}\frac{K_SL}{2}-b^2d_{mb,l}\frac{L^2}{2}}.
\end{align*}
This proves \thmref{ellr0}.\hfill\qed
\vskip\baselineskip
To deduce \corref{ellch} from \thmref{ellr0}, we note that
$\overline \chi_{-y}(S^{[n]},\mu(L))=\Ell(S^{[n]},\mu(L))\big|_{q=0}$, and by definition $\LL(f,p)\big|_{q=0}=\LL(f|_{q=0},p)$.
Thus we have 
\begin{align*}
\LL(-\phi_{-2,1},p)\big|_{q=0}&=\LL(-y+2-y^{-1},p)=\prod_{n>0}\frac{(1-p^n)^2}{(1-yp^n)(1-y^{-1}p^n)},\\
\LL^{(2,0)}(-\phi_{-2,1},p)\big|_{q=0}&=\prod_{n>0}\left(\frac{(1-p^n)^2}{(1-yp^n)(1-y^{-1}p^n)}\right)^{n^2},\\ \LL^{(1,1)}(-\phi_{-2,1},p)\big|_{q=0}&=\LL^{(1,0)}(y^{-1}-y,p)=\prod_{n>0}\left(\frac{1-y^{-1}p^n}{1-yp^n}\right)^{n},\\
\LL((\wp+h)\phi_{-2,1},p)\big|_{q=0}&=\LL(1,p)=\overline \eta(p)\quad\mathrm{and}\quad \LL(\phi_{0,1},p)\big|_{q=0}=\LL(y+10+y^{-1})=\Delta(y,p)^{1/2}.
\end{align*}
\qed

\section{The case of Hilbert schemes of points on K3 surfaces}
Now we want to consider the case of Hilbert schemes of points on K3 surfaces. We obtain a formula for all $r$.
\begin{proposition}\label{K3prop}
Let $S$ be a K3 surface and $L\in \Pic(S)$. Then 
$$\Ell(S^{[n]},\mu(L)\otimes E^r)=\Coeff_{p^n}\left[{\big(\LL^{(2,0)}(-\phi_{-2,1},p)\big)^{\frac{L^2}{2}-r^2(n-1)}\LL(-2\phi_{0,1},p)}\right].$$
\end{proposition}

\begin{proof}
The Hilbert scheme $S^{[n]}$  is a holomorphic symplectic  manifold; let $q_{S^{[n]}}$ be its Beauville-Bogolomov quadratic form. 
 By \corref{BBcor} there exists a polynomial $h_{S^{[n]}}(z)$ with coefficients in  $\Q[y^{\pm 1}][[q]] $ such that 
\begin{equation}\label{intq} \Ell(S^{[n]},M)=h_{S^{[n]}}\big(q_{S^{[n]}}(c_1(M))\big), \quad \hbox{for all }M\in \Pic(S^{[n]}).\end{equation}
It is shown in \cite[lem.~9.1] {B} that for $L\in \Pic(S)$ we have 
\begin{equation}\label{qE}
q_{S^{[n]}}(L+rE)=L^2-2r^2(n-1).\end{equation}
Therefore \propref{K3prop} follows from  \thmref{ellr0}.\end{proof}


Now we want to deduce \thmref{ellrr} from \propref{K3prop}.
We know $$\Ell(S^{[n]},\mu(L)\otimes E^r)=\int_{S^{[n]}} \ELL(S^{[n]}) \ch(\mu(L)\otimes E^r),$$
where $\ELL$ is the genus associated to a power series, and $\mu(L)\otimes E^r=\det((L+(r-1)\O_S))^{[n]}$.
Therefore  \cite[Theorem~4.2]{EGL} applies and gives the following.

\begin{corollary}\label{univers}
For every $r\in \Z$, there are universal power series $F_{r,1},F_{r,2},F_{r,3},F_{r,4}\in \Q[y^{\pm 1}][[q,p]]$ such that for every smooth projective surface $S$ and every 
$L\in\Pic(S)$ we have
$$\sum_{n\ge 0} \Ell(S^{[n]},\mu(L)\otimes E^r) p^n=F_{r,1}^{L^2/2} F_{r,2}^{LK_S/2} F_{r,3}^{K_S^2} F_{r,4}^{\chi(\O_S)}.$$
\end{corollary}

 Using \corref{univers}, in order to prove \thmref{ellrr}, we only need show the formulas for $F_{r,1}$ and $F_{r,4}$, which 
 are determined by their values for $S$ a K3 surface. 
 So let again $S$ be a K3 surface, and $L\in \Pic(S)$.
Then by \propref{K3prop} we get 
$$\sum_{n\ge 0} \Ell(S^{[n]},\mu(L)\otimes E^r)p^n =\sum_{n\ge 0} p^n\Coeff_{p^n}\left[\big(\LL^{(2,0)}(-\phi_{-2,1})(p)^{\frac{L^2}{2}-(n-1)r^2}\LL(-2\phi_{0,1})(p)\big)\right].$$
Thus \thmref{ellrr} follows from the following lemma.

\begin{lemma}\label{Lagrange} Let $g(p)=1+\sum_{i\ge 1} a_i p^i$ be a power series starting with $1$, let $f(p)$ be a power series.
Fix $w,k\in \Q$. Put 
$$h(p):=\sum_{n\ge 0} p^n\Coeff_{p^n}\big[g(p)^{w-kn}f(p)\big].$$
Then $$h(p)=\frac{g(z)^{w}f(z)}{1+kz\frac{d}{dz}\log(g(z))}\quad \hbox{where }p=zg(z)^k.$$
\end{lemma}
\begin{proof}
Without loss of generality we can assume that $w=0$ (otherwise replace $f(p)$ by $g(p)^w f(p)$), 
and $k=1$ (otherwise replace $g(p)$ by $g(p)^{k}$ and note that $z\frac{d}{dz}\log(g(z)^{k})=kz\frac{d}{dz}\log(g(p))$).
We can describe $h(p)$ as follows: for a variable $u$ write $g(p)^{-u}f(p)=\sum_{n\ge 0}h_n(u) p^n$, with $h_n(u)$ a polynomial in $u$, then $h(p)=\sum_{n\ge 0}h_n(p\frac{d}{dp}) p^n$, i.e. move all factors of $u$ to the left and then
replace $u$ by $p\frac{d}{dp}$.
We make  the variable transformation  $p=e^x$, so that $p\frac{p}{dp}=\frac{d}{dx}$, and we write $g(p)=e^{a(x)}$, $f(p)=\phi(x)$.
Then we obtain
\begin{align*}
h(p)&=
\sum_{n\ge 0} \frac{(-1)^n}{n!}\frac{d^n}{dx^n}\big(a(x)^n\phi(x)\big)\\
&=\frac{\phi(\eta)}{1+a'(\eta)}\Big|_{\eta=x-a(\eta)}\\&=\frac{f(z)}{1+z\frac{d}{dz}\log g(z)}\Big|_{p=zg(z)}.
\end{align*}
In the second line we have used the Lagrange inversion formula
$$\sum_{n=0}^\infty \frac{1}{n!}\frac{d^n}{dx^n} \big(a(x)^n f(x)\big)=\frac{f(z)}{1-a'(z)}\Big|_{z=x+a(z)}.$$
In the third line we put $z:=e^\eta$, thus $p=e^{x}=e^{\eta+a(\eta)}=zg(z)$.
\end{proof}

\section{Moduli of sheaves on K3 surfaces}

In this section we extend our results to moduli spaces of sheaves on K3 surfaces.
First we briefly recall determinant line bundles on moduli spaces of sheaves, for details see e.g. \cite[Section~1.1]{GNY}, \cite[Chapter~8]{HL}.

For a Noetherian scheme $Y$ denote by $K(Y)$ and $K^0(Y)$ the Grothendieck groups of coherent sheaves and locally free sheaves.
If $Y$ is nonsingular and projective, then $K(Y)=K^0(Y)$. We denote by $[\mathcal F]$ the class of a sheaf $\mathcal F$ in $K(Y)$.
For a proper morphism $f:Z\to Y$ the pushforward $f_!:K(Z)\to K(Y)$ is defined by $f_!([\mathcal F]):=\sum_{i}(-1)^i[R^if_*\mathcal F]$.
For any morphism $f:Z\to Y$ the pullback $f^*:K^0(Y)\to K^0(Z)$ is defined by $f^*[\mathcal F]=[f^*\mathcal \F]$ for $\F$ a locally free sheaf on $Y$.

Now let $S$ be a smooth projective surface. On $K(S)$ there is a quadratic form $(u,v)\mapsto \chi(S,u\otimes v)$ to be denoted by $\chi(u\otimes v)$. Classes
$u,v\in K(S)$ are called numerically equivalent if $u-v$ is in the radical of this quadratic form. 
Let $K(S)_{\num}$ be the set of numerical equivalence classes.
Let $c\in K(S)_{\num}$. For a flat family $\cE$ of coherent sheaves on $S$ of class $c\in K(S)_{\num}$ parametrized by a scheme $T$, let $q:S\times T\to S$, $p:S\times T\to T$ be the projections
and define
$\lambda_{\cE}:K(S)\to \Pic(S)$ by the composition
$$K(S)\xrightarrow{q^*}K^0(S\times T)\xrightarrow{\otimes [\cE]}K^0(S\times T)\xrightarrow{p_!}K^0(T)\xrightarrow{\det^{-1}}\Pic(T).$$
For $c\in K_{\num}(S)$ let $K_c:=\big\{v\in K(S)\bigm| \chi(v\otimes c)=0\big\}$. Let $H\in \Pic(S)$ be ample, denote $M^H_S(c)$ the moduli space of $H$-semistable sheaves of class $c$.
Assume that $H$ is general, i.e. if $\F\in M^H_S(c)$ is strictly $H$-semistable, then it is strictly semistable for all $H'$ in a neighbourhood of $H$. 
Then there exists a homomorphism $\lambda:K_c\to  \Pic(M^H_S(c))$, such that if $\cE$ is a flat family of coherent sheaves on $S$ of class $c$ parametrized by  $T$, 
then $\phi_{\cE}^*(\lambda(v))=\lambda_{\cE}(v)$, where $\phi_{\cE}:T\to M^H_S(c)$ is the classifying morphim associated to $\cE$.
For a class $v\in K(S)$ denote $$\vd(v):=2\rk(v)c_2(v)-(\rk(v)-1)c_1(v)^2-(\rk(v)^2-1)$$ the expected dimension of $M^H_S(v)$.
We obtain the following result.

\begin{proposition} \label{K3sheaf}
Let $S$ be a K3 surface and $c\in K_{\num}(S)$, with $r(c)>0$ or with $r(c)=0$ and $c_1(c)$ nef and big.
Let $H$ ample on $S$ such that $M^H_S(c)$ only consists of stable sheaves.
Assume furthermore $\vd(c)>1$. Then 
\begin{align*}
\Ell(M^H_S(c),\lambda(v))&=\Coeff_{p^{\frac{\vd(c)}{2}}}\left[{\big(\LL^{(2,0)}(-\phi_{-2,1},p)\big)^{\frac{\vd(v)}{2}-1}\LL(-2\phi_{0,1},p)}\right],\\
\overline \chi_{-y}(M^H_S(c),\lambda(v))&=\Coeff_{p^{\frac{\vd(c)}{2}}}\left[\frac{1}{\widetilde \Delta(p,y)}\prod_{n>0} \left(\frac{(1-p^n)^2}{(1-p^ny)(1-p^n/y)}\right)^{n^2(\frac{\vd(v)}{2}-1)}
\right].
\end{align*}
\end{proposition}
\begin{proof}
 We adapt the arguments in \cite[Section~1.5]{GNY}.
 The Mukai lattice of $S$ is $H^*(S,\Z)$ with the symmetric bilinear form 
 $$\<w,w'\>=\int_{S} (c_1c_1'-ra'-r'a), \quad w=(r,c_1,a), \ w'=(r',c_1',a').$$ 
 Let $\phi:K(S)\to H^*(S,\Z)$ be the homomorphism defined by $\phi(E)=\ch(E)\sqrt{\td(S)}.$ Let $w:=\phi(c)$, then $\phi$ induces a injective homomorphism $\phi:K_c\to w^\perp$.
 There is a homomorphism $\theta_w:w^\perp\to H^2(M_S^H(c),\Z)$, such that $\theta_w(\phi(v))=c_1(\lambda(v))$ for all $v\in K_c$.
 We have assumed  that $\rk(c)>0$ or  that $\rk(c)=0$ and $c_1(c)$ is nef and big, furthermore $\vd(c)>1$, and $M^H_S(c)$ consists only of stable sheaves. 
 Under these assumptions we know (\cite[Theorem~1.14]{GNY}, \cite{Y1}, \cite{Y2}),  that $M^H_S(c)$ is an irreducible symplectic manifold which is deformation equivalent to $S^{[\vd(c)/2]}$,
  $\theta_w:w^\perp\to  H^2(M_S^H(c),\Z)$ is surjective, and for $x$ in $w^\perp$ we have 
 $\<x,x\>=q_{M^H_S(c)}(\theta_w(x))$.
 
Let $v\in K_c$. 
By \corref{BBcor} we have
 $$\Ell(M^H_S(c),\lambda(v))=\int_{M^H_S(c)}\ELL(M^H_S(c))\exp(c_1(\lambda(v)))=h_{M^H_S(c)}\big(q_{M^H_S(c)}(\theta_w(c_1(\lambda(v))))\big).$$
 As $M^H_S(c)$ is  deformation equivalent to $S^{[\vd(c)/2]}$, we have $h_{M^H_S(c)}=h_{S^{[\vd(c)/2]}}.$
We compute 
$$q_{M^H_S(c)}\big(\theta_w(c_1(\lambda(v)))\big)=\<\phi(v),\phi(v)\>=\vd(v)-2.$$
Thus we get by \propref{K3prop} and \eqref{intq},\eqref{qE} that
$$\Ell(M^H_S(c),\lambda(v))=h_{S^{[\vd(c)/2]}}(\vd(v)-2)=\Coeff_{p^{\frac{\vd(c)}{2}}}\left[{\LL^{(2)}(-\phi_{-2,1},p)^{\frac{\vd(v)}{2}-1}\LL(-2\phi_{0,1},p)}\right].$$
 \end{proof}
  
In order to express this formula in terms of generating functions we denote the line bundles $\lambda(v)$ on different moduli spaces $M^H_S(c)$ in a unified way, generalizing  our notation on Hilbert schemes of points.
 \begin{notation}For fixed $c\in K(S)_{\num}$ we write $s=\rk(c)$, $c_1=c_1(c)$, $c_2=c_2(c)$. Let us denote by $\det(c)$ and $\det(v)\in \Pic(S)$ the determinant line bundles. We assume that $s>0$, and denote $\vd:=\vd(c)$.
 Let $r\in \Z$. Let $L\in \Pic(S)\otimes \det(c)^{r/s}$, and denote $M:=L\otimes \det(c)^{-r/s}\in \Pic(S)$.
 If $s$ divides $Mc_1+r(\frac{c_1^2}{2}-c_2)$, we define $\mu(L)\otimes E^{r}\in\Pic(M^H_S(c))\in \Pic(M^H_S(c))$ by
 $\mu(L)\otimes E^{r}:=\lambda(v)$ for $v\in K(S)$ with 
 \begin{equation}\label{vcond} \rk(v)=r, \quad c_1(v)=M, \quad c_2(v)=\frac{M^2}{2} +2r+\frac{Mc_1}{s}+\frac{r}{s}\big(\frac{c_1^2}{2}-c_2\big).\end{equation}
 Note that the condition on $c_2(v)$ is equivalent to $\chi(S,c\otimes v)=0$ i.e. to $v\in K_c$, so that $\mu(L)\otimes E^{r}$ is well-defined.
 \end{notation}
 
 \begin{remark}
 \leavevmode
 \begin{enumerate}
 \item
 When $r=0$ this definition coincides with the definition of the Donaldson line bundle $\mu(L)$ (see e.g. \cite{GNY}, \cite{GKW}). 
\item When $s=1$ the definition specializes to that of $\mu(L)\otimes E^r$ on $S^{[n]}$ under the identification $S^{[n]}=M^H_S(1,c_1,n)$, for any first Chern class $c_1$.
\item If all sheaves in $M^H_S(c)$ are slope stable, then twisting  by a line bundle $A$ gives an isomorphism $\phi_A:M^H_S(c)\to M^H_S(c\otimes A)$, and it is easy to see that   $\phi_A^*(\mu(L)\otimes E^{r})=\mu(L)\otimes E^{r}.$
\end{enumerate}
\end{remark}

Now we prove \thmref{ellK3}.
\begin{proof}[Proof of \thmref{ellK3}]Fix $c$ with $\rk(c)=s>0$, $c_1(c)=c_1$, $c_2(c)=c_2$, fix $r\in \Z$, let $L\in \Pic(S)\otimes \det(c)^{r/s}$, denote by $M:=L\otimes \det(c)^{-r/s}$ the corresponding element in $\Pic(S)$, and  assume that $s$ divides  $Mc_1+r(\frac{c_1^2}{2}-c_2)$.
Let $v\in K_c$ fullfill \eqref{vcond}, so that  $\mu(L)\otimes E^{r}=\lambda(v)$.
Plugging the relations 
$$
c_1(v)=c_1(L)-\frac{r}{s}c_1, \quad c_2(v)=\frac{c_1(v)^2}{2} +2r+\frac{c_1(v)c_1}{s}+\frac{r}{s}\big(\frac{c_1^2}{2}-c_2\big) \quad\mathrm{and}\quad c_2=\frac{1}{2s}\big(\vd(c)+(s-1)c_1^2+2(s^2-1)\big)
$$
into the formula
$$\vd(v)=2rc_2(v)-(r-1)c_1(v)^2-2(r^2-1)$$
 gives by direct computation
$$\vd(v)=L^2+2-\frac{r^2}{s^2}(\vd(c)-2).$$
Thus 
\propref{K3sheaf} gives 
\begin{align*}
\Ell(M^H_S(c),\mu(L)\otimes E^r)&=
\Coeff_{p^{\vd(c)/2}}\left[\big(\LL^{(2,0)}(-\phi_{-2,1},p)\big)^{\frac{L^2}{2}-\frac{r^2}{s^2}(\vd(c)/2-1)}\LL(-2\phi_{0,1},p)\right]
\end{align*}
Applying \lemref{Lagrange} gives \thmref{ellK3}.
\end{proof}

\ifx\undefined\bysame
\newcommand{\bysame}{\leavevmode\hbox to3em{\hrulefill}\,}
\fi

\end{document}